\documentclass[11pt]{article}
\usepackage{amssymb, amsmath, amsthm, graphicx, hyperref}
\usepackage[left=1in,top=1in,right=1in]{geometry}
\date{}
\allowdisplaybreaks

\theoremstyle{definition}
\newtheorem{theorem}{Theorem}[section]
\newtheorem{lemma}[theorem]{Lemma}

\newtheorem{proposition}[theorem]{Proposition}

\newtheorem{claim}[theorem]{Claim}

\counterwithin{equation}{section}

\title{On cliques in three-dimensional dense point-line arrangements}
\author{Andrew Suk\thanks{Department of Mathematics, University of California San Diego, La Jolla, CA, 92093 USA. Supported by NSF CAREER award DMS-1800746 and NSF award DMS-1952786. Email: {\tt asuk@ucsd.edu}.} \and Ji Zeng\thanks{University of California San Diego, La Jolla, CA and Alfr\'ed R\'enyi Institute of Mathematics, Budapest, Hungary. Supported by NSF grant DMS-1800746. Partially supported by ERC grant No. 882971, ``GeoScape'', and by the Erd\H os Center. Email:{\tt jzeng@ucsd.edu}.}}

\begin{document}

\maketitle

\begin{abstract}
As a variant of the celebrated Szemer\'edi--Trotter theorem, Guth and Katz proved that $m$ points and $n$ lines in $\mathbb{R}^3$ with at most $\sqrt{n}$ lines in a common plane must determine at most $O(m^{1/2}n^{3/4})$ incidences for $n^{1/2}\leq m\leq n^{3/2}$. This upper bound is asymptotically tight and has an important application in the Erd{\H o}s distinct distances problem. We characterize the extremal constructions towards the Guth--Katz bound by proving that such a large dense point-line arrangement must contain a $k$-clique in general position provided $m \ll n$. This is an analog of a result by Solymosi for extremal Szemer\'edi--Trotter constructions in the plane.
\end{abstract}

\section{Introduction}\label{sec_intro}

A \textit{point-line arrangement} is a pair of a point set $P$ and a line set $L$ both in a common Euclidean space. An \text{incidence} of $(P,L)$ is a pair $(p,\ell) \in P\times L$ such that $p\in \ell$. We use $I(P, L)$ to denote the collection of all incidences of $(P, L)$. The celebrated Szemer\'edi--Trotter theorem \cite{szemeredi1983extremal} states that any arrangement $(P,L)$ in the Euclidean plane $\mathbb{R}^2$ satisfies\begin{equation}\label{szemereditrotter}
    |I(P,L)| < O\left(|P|^{\frac{2}{3}}|L|^{\frac{2}{3}} + |P| + |L|\right).
\end{equation} This bound is asymptotically tight in the worst case, and many results suggest that the extremal constructions towards the Szemer\'edi--Trotter bound have special structures \cite{solymosi2006dense,mirzaei2021grids,suk2021hasse,sheffer2023structural,katz2023structure,dasu2023structural}. In 2006, Solymosi \cite{solymosi2006dense} proved that a large dense arrangement in $\mathbb{R}^2$ must contain a $k$-clique in general position. A \textit{$k$-clique} of $(P, L)$ refers to a $k$-subset of $P$ with each pair incident to a common line of $L$. A $k$-clique $\tau$ in $\mathbb{R}^d$ is said to be in \textit{general position} if no hyperplane contains $d+1$ points of $\tau$.
\begin{theorem}[Solymosi]\label{solymosi}
    For every integer $k \geq 3$ and real $c > 0$, there exists a constant $n_0 = n_0(k,c)$ such that if an arrangement $(P,L)$ of $n$ points and $n$ lines in $\mathbb{R}^2$ satisfies $n \geq n_0$ and $|I(P,L)| > c n^{4/3}$, then $(P,L)$ contains a $k$-clique in general position.
\end{theorem}
\noindent Following Solymosi's argument, it's easy to generalize this result to large dense arrangements of $m$ points and $n$ lines provided $n^{1/2}\lesssim m\lesssim n^2$.

For a point-line arrangement in any higher dimension, we can project it onto a generic (two-dimensional) plane so that \eqref{szemereditrotter} still holds. In the breakthrough paper by Guth and Katz \cite{guth2015erdHos} on the Erd{\H o}s distinct distances problem, an improved Szemer\'edi--Trotter-type bound (see also Theorem~12.1 in \cite{guth2016polynomial}) was derived for arrangements in $\mathbb{R}^3$.
More specifically, they showed that any arrangement $(P,L)$ in $\mathbb{R}^3$ with at most $B$ lines of $L$ in a common plane satisfies\begin{equation}\label{guthkatz}
    |I(P,L)| < O\left(|P|^{\frac{1}{2}}|L|^{\frac{3}{4}} + B^{\frac{1}{3}}|P|^{\frac{2}{3}}|L|^{\frac{1}{3}}+|P|+|L|\right).
\end{equation} This bound is also asymptotically tight when $B = \sqrt{|L|}$, see Example~3 in \cite{guth2015erdHos} and our Section~\ref{sec_lower}.

Inspired by the theorem of Solymosi \cite{solymosi2006dense}, we prove the following characterization of extremal constructions towards the Guth--Katz bound.
\begin{theorem}\label{main1}
    For every integer $k \geq 4$ and real $c > 0$, there exist constants $n_0 = n_0(k,c)$ and $\delta = \delta(k,c)$ such that if an arrangement $(P,L)$ of $m$ points and $n$ lines in $\mathbb{R}^3$ satisfies: $n_0< n^{1/2} \leq m \leq \delta n$; $|I(P,L)| > c m^{1/2}n^{3/4}$; and at most $\sqrt{n}$ lines of $L$ lie in a common plane, then $(P,L)$ contains a $k$-clique in general position.
\end{theorem}
\noindent The ``truly three-dimensional'' condition (i.e.~there are at most $\sqrt{n}$ coplanar lines) is a natural requirement so that \eqref{guthkatz} is dominated by the $|P|^{1/2}|L|^{3/4}$ term. However, the ``$|P|\ll |L|$'' condition is due to a technical obstacle in our proof. We remark that the latter condition can be removed if one merely wishes to find a $k$-clique without collinear triples, but here, $k$-cliques in general position also forbid coplanar quadruples.

Using the probabilistic method, it's not hard to construct point-line arrangements with many incidences and no $k$-cliques in general position. 
\begin{theorem}\label{lowerbound2d}
    For integers $k \geq 4$ and $n \to \infty$, there exists an arrangement $(P,L)$ of $n$ points and $n$ lines in $\mathbb{R}^2$ with no $k$-cliques in general position and $|I(P,L)| > n^{4/3 - O(1/k)}$.
\end{theorem}
\begin{theorem}\label{lowerbound3d}
    For integers $k \geq 4$ and $m,n\to\infty$ with $n^{1/2}\leq m\leq n$, there exists an arrangement $(P,L)$ of $m$ points and $n$ lines in $\mathbb{R}^3$ with no $k$-cliques in general position, at most $\sqrt{n}$ lines of $L$ in a common plane, and $|I(P,L)| > m^{1/2 - O(1/k)} n^{3/4}$.
\end{theorem}

Szemer\'edi--Trotter-type bounds can often be restated in terms of upper bounds on rich points. For a line set $L$ in a Euclidean space, the \textit{$r$-rich points} of $L$, denoted by $P_r(L)$, are points incident to at least $r$ lines in $L$. For example, the classical Szemer\'edi--Trotter bound \eqref{szemereditrotter} is equivalent to $|P_r(L)| < O(|L|^2/r^3 + |L|/r)$. Guth and Katz (Theorem~2.11 in \cite{guth2015erdHos}) showed that any line set $L$ in $\mathbb{R}^3$ with at most $\sqrt{|L|}$ lines of $L$ in a common plane satisfies \begin{equation} \label{guthkatzrich}
    |P_r(L)| < O\left(|L|^{\frac{3}{2}}/r^2\right) \quad\text{for}\quad 3\leq r \leq |L|^{1/2}.
\end{equation} As remarked in Chapter 12 of \cite{guth2016polynomial}, \eqref{guthkatz} only implies \eqref{guthkatzrich} for $r_0\leq r\leq |L|^{1/2}$, where $r_0$ depends on the constant hidden in the $O$-notation of \eqref{guthkatz}. (Extra work needs to be taken to prove \eqref{guthkatzrich} for smaller $r$.)

In view of Theorem~\ref{main1}, we have the following bound on rich points. The function ``$\log(n)$'' below can be replaced with any function that grows to infinity.
\begin{theorem}\label{main2}
    For every integer $k \geq 4$, if $L$ is a set of $n$ lines in $\mathbb{R}^3$ with at most $\sqrt{n}$ lines of $L$ in a common plane, and if the arrangement $(P_r(L),L)$ contains no $k$-cliques in general position for some integer $r$ with $n^{1/4}\cdot\log(n)\leq r \leq n^{1/2}/\log(n)$, then $|P_r(L)| < o(n^{3/2}/r^2)$.
\end{theorem} The upper bound \eqref{guthkatzrich} is asymptotically tight for $2\leq r \ll |L|^{1/2}$ (see Example~3 in \cite{guth2015erdHos}). A probabilistic argument on the extremal constructions towards \eqref{guthkatzrich} can give us line sets $L$ satisfying the hypothesis of Theorem~\ref{main2} for a slightly smaller range of $r$, but with $|P_r(L)| \gg n^{3/2 - O(1/k)}/r^2$. We omit the proof as it's similar to that of Theorem~\ref{lowerbound3d}.

The rest of this paper is organized as follows. In Section~\ref{sec_pre}, we present notations and results about the regularity method and polynomial partitioning. In Section~\ref{sec_upper}, we prove Theorem~\ref{main1} and Theorem~\ref{main2} using tools in the previous section and the so-called same-type lemma by B{\'a}r{\'a}ny and Valtr~\cite{barany1998positive}. Section~\ref{sec_lower} is devoted to the lower bound constructions stated in Theorem~\ref{lowerbound2d} and Theorem~\ref{lowerbound3d}. We discuss related open problems in Section~\ref{sec_remark}. We omit floors and ceilings whenever they are not crucial for the sake of clarity in our presentation. The functions $\exp(x)$ and $\log(x)$ are both in base $e$.

\section{Preliminaries}\label{sec_pre}

    We shall use Szemer\'edi's regularity lemma and graph counting lemma in our proof. We briefly list notations and statements of them, and the readers are referred to \cite{zhao2023graph} for a detailed introduction. Let $X$ and $Y$ be sets of vertices in a graph $G$.  Then the \textit{density} of the pair $(X,Y)$ is defined as\begin{equation*}
        d(X,Y) = \dfrac{|\{(x,y)\in X\times Y ~|~ \{x,y\}\in E(G)\}|}{|X|\cdot|Y|}.
    \end{equation*}
    We say that $(X,Y)$ is an \textit{$\epsilon$-regular pair} if for all $A\subset X$ and $B\subset Y$ with $|A|\geq \epsilon|X|$ and $|B|\geq \epsilon|Y|$, we have $|d(X,Y) - d(A,B)|\leq \epsilon$. A partition $\mathcal{P}=\{V_1,V_2,\dots,V_k\}$ of the vertex set of $G$ is said to be an \textit{$\epsilon$-partition} if\begin{equation*}
        \sum_{\substack{1\leq i,j\leq k\\ (V_i,V_j)\text{ not $\epsilon$-regular}}} |V_i|\cdot|V_j| \leq \epsilon |V(G)|^2.
    \end{equation*} Moreover, a partition $\mathcal{P}$ is said to be \textit{equitable} if all part sizes are within one of each other. We have the following two lemmas.
    \begin{lemma}[Theorem~2.1.17 in \cite{zhao2023graph}]\label{reglem}
         For all $\epsilon > 0$ and $M_0$, there exists a constant $M$ such that every graph has an $\epsilon$-regular equitable partition of its vertex set into at least $M_0$ and at most $M$ parts.
     \end{lemma}
    \begin{lemma}[Theorem~2.6.4 in \cite{zhao2023graph}]\label{graphcount}
        For every integer $k \geq 1$ and real $\epsilon > 0$, let $G$ be a graph and $V_1,V_2,\dots,V_k$ be pairwise disjoint subsets of $V(G)$. Suppose for each $1\leq i<j\leq k$, $(V_i,V_j)$ is an $\epsilon$-regular pair with density $d(V_i,V_j) \geq (k+1)\epsilon^{1/k}$, then the number of $k$-cliques $v_1,v_2,\dots,v_k$ inside $G$ with $v_i\in V_i$ for each $i$ is at least\begin{equation*}
            (1-k\epsilon) \cdot \prod_{1\leq i<j\leq k} \left(d(V_i,V_j) - k\epsilon^{\frac{1}{k}}\right) \cdot \prod_{1\leq i\leq k} |V_i|.
        \end{equation*}
     \end{lemma}

     Another ingredient in our proof is the polynomial partitioning lemma, see e.g.,~Theorem~10.3 in \cite{guth2016polynomial}.
    \begin{lemma}\label{polypar}
        For a finite point set $P\subset \mathbb{R}^3$ and a positive integer $D$, there exists a nonzero polynomial $f\in \mathbb{R}[x,y,z]$ of degree at most $D$ such that $\mathbb{R}^3 \setminus Z(f)$ is a disjoint union of $O(D^3)$ open sets each containing $O(|P|/D^3)$ points of $P$.
    \end{lemma}

    We need the following estimate of incidences on algebraic surfaces essentially due to Guth and Katz. We shall briefly explain how to obtain this bound by piecing together arguments scattered in the proof of \eqref{guthkatz} in \cite{guth2015erdHos}.
    \begin{proposition}\label{algebraic_incidence_estimate}
        Suppose $f\in \mathbb{R}[x,y,z]$ is a polynomial of degree $D$, $L$ is a set of $n$ lines in $\mathbb{R}^3$ with at most $B$ of its members on a common plane, and $P$ is a set on $m$ points on the variety $Z(f)$. Then we have \begin{equation*}
        |I(P,L)| < O\left(Dn+ B^{\frac{1}{3}}m^{\frac{2}{3}}n^{\frac{1}{3}}+D^{\frac{2}{3}}B^{\frac{1}{3}}m^{\frac{2}{3}}+ m +D^{\frac{3}{2}}m^{\frac{1}{2}} + D^2\right).
        \end{equation*}
    \end{proposition}
    \begin{proof}[Proof Sketch]
             Let $L_{alg}\subset L$ be the lines contained in $Z(f)$. Since any line not contained in $Z(f)$ has at most $D$ intersections with $Z(f)$, we have\begin{equation*}
                 |I(P, L\setminus L_{alg})| \leq Dn.
             \end{equation*} As a consequence, it suffices for us to estimate $I(P, L_{alg})$.
    
             We call a plane that is a subset of $Z(f)$ an \textit{algebraic plane}. Let $L_{lin}\subset L_{alg}$ be the lines contained in at least one algebraic plane, and $L_{uni}\subset L_{lin}$ be the lines contained in exactly one algebraic plane. One can apply \eqref{szemereditrotter} on each algebraic plane to obtain the following bound (see Lemma~12.6 in \cite{guth2016polynomial} for details).
             \begin{equation*}
                 |I(P, L_{uni})| < O\left(Dn + B^{\frac{1}{3}}m^{\frac{2}{3}}n^{\frac{1}{3}}\right).
             \end{equation*}
    
             Let $P' \subset P$ be the points that are either \textit{flat} or \textit{critical} on $Z(f)$, and $L'\subset L_{alg}$ be the lines consisting of flat or critical points on $Z(f)$. The readers are referred to Chapter~11 in \cite{guth2016polynomial} for the definition and properties of flat points. Using these special properties, we have the following bounds (see Lemma~12.6 in \cite{guth2016polynomial} for details).
             \begin{equation*}
                 |I(P \setminus P', L_{alg}\setminus L')| \leq 2m \quad\text{and}\quad |I(P', L_{alg} \setminus L')| \leq 3Dn.
             \end{equation*}
    
             Moreover, Proposition~12.9 in \cite{guth2016polynomial} states that $|L' \setminus L_{lin}| \leq 4D^2$. There are at most $D$ algebraic planes in $Z(f)$ and two planes determine at most one line, so we have $|L_{lin}\setminus L_{uni}| \leq D^2$. Therefore, we can apply the Guth--Katz bound \eqref{guthkatz} to obtain\begin{equation*}
                 |I(P, L' \setminus L_{lin})| + |I(P,L_{lin}\setminus L_{uni}))| < O\left(m^{\frac{1}{2}}D^{\frac{3}{2}} + B^{\frac{1}{3}}m^{\frac{2}{3}}D^{\frac{2}{3}}+m+D^2\right).
             \end{equation*}
    
            Finally, we can break down $|I(P, L)|$ as follows.\begin{align*}
                 |I(P,L)| = |I(P, L\setminus L_{alg})|+ |I(P, L_{alg}\setminus L')| &+ |I(P, L'\setminus L_{lin})| + |I(P , L_{lin}\setminus L_{uni})| +|I(P , L_{uni})|\\
                 = |I(P, L\setminus L_{alg})|+ |I(P\setminus P', L_{alg}\setminus L')| &+ |I(P', L_{alg}\setminus L')|\\
                 &\hspace{-2em}+ |I(P, L'\setminus L_{lin})| + |I(P , L_{lin}\setminus L_{uni})| +|I(P , L_{uni})|.
             \end{align*} And we conclude the proof by combining our identities above.
         \end{proof}

\section{Proofs of Theorem~\ref{main1} and Theorem~\ref{main2}}\label{sec_upper}

We prove the following more general statement which implies Theorem~\ref{main1} and Theorem~\ref{main2}.
\begin{theorem}\label{main}
    For every integer $k \geq 4$ and real $c > 0$, there exists a constant $\delta = \delta(k,c)$ such that if an arrangement $(P,L)$ of $m$ points and $n$ lines in $\mathbb{R}^3$ satisfies: $n^{1/2} \leq m \leq n$; $|I(P,L)| > c m^{1/2}n^{3/4}$; any plane contains at most $\delta m^{-1}n^{3/2}$ lines in $L$, then $(P,L)$ contains $\Omega\left((m^{-1/2}n^{3/4})^k\right)$ many $k$-cliques in general position. Here, the constant hidden in the $\Omega$-notation depends on $k$ and $c$.
\end{theorem}
\begin{proof}
     First, we apply Lemma~\ref{polypar} to the point set $P$ with degree $D = c_1 m^{1/2}n^{-1/4}$, where $c_1 = c_1(k,c)$ is a constant we shall determine later. This gives us a polynomial $f$ of degree $D$ such that $\mathbb{R}^3\setminus Z(f)$ is a disjoint union of at most $O(D^3)$ open cells, denoted as $O_i$, and each $O_i$ contains at most $O(m/D^3)$ points from $P$. Then we apply Proposition~\ref{algebraic_incidence_estimate} and the hypothesis $n^{\frac{1}{2}}\leq m \leq n$ to the point set $P \cap Z(f)$, and this gives us \begin{align*}
         |I(P \cap Z(f),L)| &< O\left(Dn+ \left(\delta m^{-1}n^{3/2}\right)^{\frac{1}{3}}m^{\frac{2}{3}}n^{\frac{1}{3}}+D^{\frac{2}{3}}\left(\delta m^{-1}n^{3/2}\right)^{\frac{1}{3}}m^{\frac{2}{3}}+ m +D^{\frac{3}{2}}m^{\frac{1}{2}} + D^2\right)\\
         &< O\left(\left(c_1 + \delta^{\frac{1}{3}}\right)m^{\frac{1}{2}}n^{\frac{3}{4}}\right).
     \end{align*} By taking $\delta$ and $c_1$ to be small enough, we can guarantee that \begin{equation}\label{eq_choice1}
         |I(P \cap Z(f),L)|<O\left(\left(c_1 + \delta^{\frac{1}{3}}\right)m^{\frac{1}{2}}n^{\frac{3}{4}}\right)< \frac{c}{2}m^{\frac{1}{2}}n^{\frac{3}{4}} < \frac{1}{2}|I(P,L)|.
     \end{equation} Hence, $|I(P\setminus Z(f),L)| > |I(P,L)|/2 > cm^{1/2}n^{3/4}/2$.
    
    Next, we estimate the number of collinear $k$-tuples in $P$. Let $L \pitchfork O_i$ denote the set of lines $\ell\in L$ with $\ell\cap O_i \neq \emptyset$. Because any line intersects at most $D$ cells, we have $\sum_i |L \pitchfork O_i| \leq Dn$. Since $\lfloor x \rfloor \geq x-1$ and $\lfloor 0 \rfloor =0$, we can compute \begin{equation*}
        \sum_{O_i} \sum_{\ell\in L} \left\lfloor \frac{|\ell \cap P \cap O_i|}{k} \right\rfloor \geq \sum_{O_i} \sum_{\ell} \frac{|\ell \cap P \cap O_i|}{k} - \sum_{O_i} |L \pitchfork O_i| \geq \frac{|I(P\setminus Z(f), L)|}{k} - Dn.
    \end{equation*} By the inequalities above and taking $c_1$ to be small, we can guarantee that\begin{equation}\label{eq_choice2}
        \sum_{O_i} \sum_{\ell\in L} \left\lfloor \frac{|\ell \cap P \cap O_i|}{k} \right\rfloor \geq \frac{c}{2k} m^{\frac{1}{2}}n^{\frac{3}{4}} - Dn > \frac{c}{3k} m^{\frac{1}{2}}n^{\frac{3}{4}}.
    \end{equation} According to the pigeonhole principle and Lemma~\ref{polypar}, there exists $i_0$ such that \begin{equation*}
        \sum_{\ell\in L} \left\lfloor \frac{|\ell \cap P \cap O_{i_0}|}{k} \right\rfloor  \geq \Omega\left( \frac{c}{3k} m^{\frac{1}{2}}n^{\frac{3}{4}} / D^3 \right) = \Omega\left(\frac{c}{kc_1^3} m^{-1}n^{\frac{3}{2}}\right).
    \end{equation*}
    
    Now, we construct an auxiliary graph $G$ using the point-line arrangement in $O_{i_0}$. We define the vertex set $V(G) = P\cap O_{i_0}$ and write $N = |V(G)|$. Let $L_r$ be the set of lines $\ell\in L$ with $|\ell \cap V(G)| \geq r$. We define the edge set $E(G)$ to be the pairs coincident to any line in $L_k\setminus L_R$. Here, $R =R(k,c)$ is a constant we shall soon determine. By Lemma~\ref{polypar}, $N \leq O(m/D^3) = c_2 m^{-1/2}n^{3/4}$ for some constant $c_2=c_2(k,c)$. Applying the previous inequality, we can compute\begin{equation*}
        |I(V(G), L_k)| \geq k\cdot\sum_{\ell\in L} \left\lfloor \frac{|\ell \cap P \cap O_{i_0}|}{k} \right\rfloor  > \Omega\left(\frac{c}{c_1^3} m^{-1}n^{\frac{3}{2}}\right) \geq \Omega\left(\frac{c}{c_1^3c_2^2} N^2\right).
    \end{equation*} On the other hand, the Szemer\'edi--Trotter bound \eqref{szemereditrotter} tells us that \begin{equation*}
        |I(V(G), L_k)| < O\left(N^{2/3}|L_k|^{2/3} + N + |L_k|\right) \quad\text{and}\quad |L_r| < O\left({N^2}/{r^3} + {N}/{r}\right).
    \end{equation*} From $N^{2/3}|L_k|^{2/3} + N + |L_k| \gtrsim N^2$, we can solve for $|L_k| \gtrsim N^2$. Hence by taking a large enough $R$, we can guarantee that \begin{equation}\label{eq_choice3}
        |L_k \setminus L_R| \geq c_3 N^2,
    \end{equation} for some constant $c_3=c_3(k,c)$. Note that \eqref{eq_choice3} implies the existence of $c_3 N^2$ edge disjoint $k$-cliques in $G$. Additionally, the trivial bound $|E(G)| \leq |V(G)|^2$ implies that\begin{equation*}
       c_4 m^{-\frac{1}{2}}n^{\frac{3}{4}} \leq N \leq c_2 m^{-\frac{1}{2}}n^{\frac{3}{4}},
    \end{equation*} for some constant $c_4= c_4(k,c)$.

    Our proof continues by applying the regularity lemma on $G$. By Lemma~\ref{reglem}, for two constants $\epsilon=\epsilon(k,c)$ and $M_0=M_0(k,c)$ to be determined later, there is an equitable $\epsilon$-regular partition of $G$ into parts $V_1,\dots,V_M$ such that $M>M_0$. With foresight, we define a constant $K=3^{2k^3}$ (whose importance shall appear soon in Claim~\ref{sametypeclaim}). For each $1\leq i,j\leq M$, we remove all the edges in $G$ between $V_i$ and $V_j$ if either of the following happens: the indices $i = j$; or the pair $(V_i,V_j)$ isn't $\epsilon$-regular; or the density $d(V_i,V_j) < (k+2)(K\epsilon)^{1/k}$. We can estimate the number of deleted edges: at most $N^2/M_0$ edges are deleted due to ``same index'' as the partition is equitable and $M > M_0$; at most $\epsilon N^2$ edges are deleted due to ``irregularity'' as the partition is $\epsilon$-regular; at most $(k+2)(K\epsilon)^{1/k} N^2$ edges are deleted due to ``low density'' by the definition of density. Hence, we can choose $\epsilon$ and $1/M_0$ to be small enough such that\begin{equation}\label{eq_choice4}
        \text{\# of deleted edges} \leq \left(\frac{1}{M_0} + \epsilon + (k+2)(K\epsilon)^{\frac{1}{k}}\right) N^2 < c_3 N^2.
    \end{equation} As a consequence of \eqref{eq_choice3} and \eqref{eq_choice4}, the graph after edge-deletion still contains a $k$-clique. This means there are $k$ parts, (without loss of generality) $V_1,V_2,\dots,V_k$, each pair among which is $\epsilon$-regular and has density at least $(k+2)(K\epsilon)^{1/k}$.

    Next, we shall prune the parts $V_1,\dots, V_k$ such that their transversals are in general position. More precisely, we have the following claim.  Recall that any plane contains at most $\delta m^{-1}n^{3/2}$ lines in $L$.
    \begin{claim}\label{sametypeclaim}
        For a suitable choice of $\delta$, there exist subsets $V_i'\subset V_i$ such that, for $1\leq i\leq k$, we have $|V_i'| \geq |V_i|/K$ and every $\{v_1,v_2,\dots,v_k\}$ with $v_i\in V_i'$ is in general position.
    \end{claim}
    \noindent We suppose this claim is true and let $V_i'$ be as given. For every $1\leq i,j\leq k$, since $|V_i'| \geq |V_i|/K$ and $(V_i,V_j)$ is $\epsilon$-regular, we can check that $(V_i',V_j')$ is $(K\epsilon)$-regular (see Exercise~2.1.4 in \cite{zhao2023graph}). Notice that $\epsilon < 1/K$ as the coefficient of $N^2$ in \eqref{eq_choice4} is less than $1$. So we also have $|d(V_i,V_j) - d(V_i',V_j')| \leq \epsilon$, which implies\begin{equation*}
        d(V_i',V_j') \geq (k+2)(K\epsilon)^{1/k} - \epsilon > (k+1)(K\epsilon)^{1/k}.
    \end{equation*} Therefore, Lemma~\ref{graphcount} implies that there exists $c_6 |V_i'|^k$ many $k$-cliques $v_1, v_2, \dots, v_k$, with $v_i\in V_i'$, where $c_6=c_6(k,c)$ is some constant. Moreover, each such $k$-clique is in general position by the claim above. So we have the desired lower bound by recalling that $|V_i'| \geq N/(MK)$ and $N \geq c_4 m^{-\frac{1}{2}}n^{\frac{3}{4}}$.

    It suffices for us to prove Claim~\ref{sametypeclaim}. Our argument is inspired by the same-type lemma (see Theorem~9.3.1 in \cite{matousek2013lectures}).
    \begin{proof}[Proof of Claim~\ref{sametypeclaim}]
        We let $\{A_1,B_1\},\{A_2,B_2\},\dots$ be an enumeration of all pairs $\{A_\alpha,B_\alpha\}$ of nonempty subsets of $\{1,2,\dots,k\}$ such that $A_\alpha\cap B_\alpha=\emptyset$ and $|A_\alpha\cup B_\alpha|=4$. Here, it's obvious that $\alpha$ ranges from $1$ to $7\binom{k}{4}$. Let $U^0_i = V_i$ for all $1\leq i\leq k$ and we shall iteratively construct $U^\alpha_i \subset U^{\alpha-1}_i$ satisfying \begin{itemize}
            \item $|U^{\alpha}_i| \geq |U^{\alpha-1}_i|/3$ for $i\in A_{\alpha}\cup B_\alpha$ and $U^\alpha_i = U^{\alpha-1}_i$ for $i\not\in A_\alpha\cup B_\alpha$, and
            \item the unions of convex hulls $\bigcup_{i\in A_\alpha} \text{conv}(U^{\alpha}_i)$ and $\bigcup_{j\in B_\alpha} \text{conv}(U^{\alpha}_j)$ are strictly separated by a plane $\pi_\alpha$.
        \end{itemize} Then we can take $V_i' = U^\alpha_i$ for $1\leq i\leq k$ and $\alpha=7\binom{k}{4}$. For each particular $1\leq i \leq k$, it's obvious that $i \in A_\alpha\cup B_\alpha$ happens for at most $7\binom{k-1}{3} < 2k^3$ times, so we have $|V_i'|\geq |V_i|/K$ by the first property of this iterative process. Moreover, we argue that any $\{v_1,v_2,\dots,v_k\}$ with $v_i\in V_i'$ is in general position. Indeed, if four points $v_{i_1},v_{i_2},v_{i_3},v_{i_4}$ lie on the same plane, this plane should simultaneously intersect $\text{conv}(V_{i_1})$, $\text{conv}(V_{i_2})$, $\text{conv}(V_{i_3})$, and $\text{conv}(V_{i_4})$. However, this will contradict the second property of the iterative process and the following fact (Lemma~9.3.2 in \cite{matousek2013lectures}).
        \begin{lemma}
            Let $C_1,\dots,C_{d+1} \subset \mathbb{R}^{d}$ be convex sets. The following two conditions are equivalent:\begin{enumerate}
                \item There is no hyperplane simultaneously intersecting all of $C_1,C_2,\dots,C_{d+1}$.
                \item For each nonempty index set $I \subset \{1,2,\dots,d+1\}$, the sets $\bigcup_{i\in I} C_i$ and $\bigcup_{j\not\in I} C_j$ can be strictly separated by a hyperplane.
            \end{enumerate}
        \end{lemma}
        
        Now it suffices for us to specify the iterative construction at each step. Here, we describe the case when $A_\alpha=\{1,2\}$ and $B_\alpha=\{3,4\}$, and general index combinations can be dealt with in a similar fashion. By the Ham-Sandwich theorem (Theorem~1.4.3 in \cite{matousek2013lectures}), there is a plane $\pi$ bisecting $U^{\alpha-1}_{1},U^{\alpha-1}_{2},U^{\alpha-1}_{3}$ simultaneously. Here, ``bisecting'' means each open half-space determined by this plane intersects at most half of the point set. Note that there exists one half-space $H$ determined by $\pi$ such that its closure $\bar{H}$ contains half of $U^{\alpha-1}_{4}$, i.e. $|\bar{H} \cap U^{\alpha-1}_{4}| \geq |U^{\alpha-1}_{4}|/2$.
        
        We show that, provided $\delta$ is small enough, there is at most one index $1\leq i\leq 4$ such that \begin{equation}\label{eq_choice5}
            |\pi \cap U^{\alpha-1}_{i}| \geq |U^{\alpha-1}_{i}|/6.
        \end{equation} On the contrary, we suppose without loss of generality that \eqref{eq_choice5} holds for $i=1$ and $i=2$. Then we have $|\pi \cap U^{\alpha-1}_{i}| \geq |V_i|/(6K)$ for $i=1,2$, because $|U^{\alpha-1}_{i}| \geq |V_i|/K$ is guaranteed throughout the iterative process. Notice that \eqref{eq_choice4} and $k \geq 4$ implies $\epsilon < 1/(6K)$. As the pair $(V_1,V_2)$ is $\epsilon$-regular, we have $|d(\pi \cap U^{\alpha-1}_{1}, \pi \cap U^{\alpha-1}_{2}) -  d(V_1,V_2)| \leq \epsilon$. Then we can compute, using the bounds appearing above and the definition of density,\begin{equation*}
            |E(\pi \cap U^{\alpha-1}_{1}, \pi \cap U^{\alpha-1}_{2})| \geq \left((k+2)(K\epsilon)^{1/k} - \epsilon\right) \frac{|V_1|}{6K} \frac{|V_2|}{6K} = c_5 N^2 \geq c_5c_4^2 m^{-1}n^{\frac{3}{2}},
        \end{equation*} for some constant $c_5=c_5(k,c)$. Notice that each line $\ell \in L_k\setminus L_R$ contributes at most $R^2$ edges in $G$. So we can choose $\delta$ to be small enough such that \begin{equation*}
            \text{\# of lines on $\pi$} \geq |E(\pi \cap U^{\alpha-1}_{1}, \pi \cap U^{\alpha-1}_{2})|/R^2 > c_5c_4^2R^{-2} m^{-1}n^{\frac{3}{2}} >\delta m^{-1}n^\frac{3}{2}.
        \end{equation*} However, this is a contradiction to our hypothesis that any plane contains at most $\delta m^{-1}n^{3/2}$ lines in $L$. Hence the inequality \eqref{eq_choice5} holds for at most one index $i$.

        If \eqref{eq_choice5} holds for no $i$, then each half-space of $\pi$ contains one third of $U^{\alpha-1}_{i}$ for each $i=1,2,3$, and $H$ contains one third of $U^{\alpha-1}_{4}$. Then we can set $U_i^\alpha = U_i^{\alpha-1} \setminus \bar{H}$ for $i=1,2$ and $U_i^\alpha = U_i^{\alpha-1}\cap H$ for $i=3,4$. In this way, we have $|U^{\alpha}_i| \geq |U^{\alpha-1}_i|/3$ for $1\leq i\leq 4$, and the plane $\pi_\alpha =\pi$ strictly seperates $\text{conv}(U^{\alpha}_1) \cup \text{conv}(U^{\alpha}_2)$ and $\text{conv}(U^{\alpha}_3) \cup \text{conv}(U^{\alpha}_4)$. Hence we conclude the iterative step.

        If \eqref{eq_choice5} holds for $i=1$, then each half-space of $\pi$ contains one third of $U^{\alpha-1}_{i}$ for each $i=2,3$, and $H$ contains one third of $U^{\alpha-1}_{4}$. We consider a plane $\pi_\alpha$ parallel to $\pi$ but inside $H$. Denote the half-spaces of $\pi_\alpha$ as $H_1$ and $H_2$ such that $H_1 \not\subset H$ and $H_2 \subset H$. We can choose $\pi_\alpha$ to be very close to $\pi$ so that $H_2$ still contains one third of $U^{\alpha-1}_{i}$ for each $i=3,4$. Notice that $H_1$ contains half of $U^{\alpha-1}_{i}$ for $i = 1,2$ by the ``bisecting'' property of $\pi$. Hence, the iterative step can be done by setting $U_i^\alpha = U_i^{\alpha-1} \cap H_1$ for $i=1,2$ and $U_i^\alpha = U_i^{\alpha-1}\cap H_2$ for $i=3,4$.

        If \eqref{eq_choice5} holds for $i=4$, then each half-space of $\pi$ contains one third of $U^{\alpha-1}_{i}$ for each $i=1,2,3$. We consider a plane $\pi_\alpha$ parallel to $\pi$ but outside $H$. Denote the half-spaces of $\pi_\alpha$ as $H_1$ and $H_2$ such that $H_1 \cap H =\emptyset$ and $\bar{H} \subset H_2$. We can choose $\pi_\alpha$ to be very close to $\pi$ so that $H_1$ still contains one third of $U^{\alpha-1}_{i}$ for each $i=1,2$. Notice that $H_2$ contains half of $U^{\alpha-1}_{1}$ by $\bar{H} \subset H_2$, and half of $U^{\alpha-1}_{3}$ by the ``bisecting'' property. Hence, the iterative step can be done by setting $U_i^\alpha = U_i^{\alpha-1} \cap H_1$ for $i=1,2$ and $U_i^\alpha = U_i^{\alpha-1}\cap H_2$ for $i=3,4$.

        If \eqref{eq_choice5} holds for $i=2$, the iterative step can be similarly performed as when it holds for $i=1$. If \eqref{eq_choice5} holds for $i=3$, the iterative step can be similarly performed as when it holds for $i=4$. Hence we conclude the proof of Claim~\ref{sametypeclaim}.
    \end{proof}

    Finally, we remark that the parameters can be chosen properly. We first choose $c_1$ according to \eqref{eq_choice1} and \eqref{eq_choice2}. To satisfy \eqref{eq_choice1}, we assume $\delta$ is smaller than a fixed value. After $c_1$ is determined, we determine $R,\epsilon,M_0$ according to \eqref{eq_choice3} and \eqref{eq_choice4}. Then we choose $\delta$ to be further small enough so that our claim for \eqref{eq_choice5} holds. Hence we conclude the proof of Theorem~\ref{main}.
\end{proof}

\begin{proof}[Proof of Theorem~\ref{main1}]
    We apply Theorem~\ref{main} with $k$ and $c$ as given in our hypothesis, then we obtain a $\delta = \delta(k,c)$ such that if an arrangement $(P,L)$ of $m$ points and $n$ lines in $\mathbb{R}^3$ satisfies $|I(P,L)| > c n^{5/4}$ and any plane contains at most $\delta m^{-1}n^{3/2}$ lines in $L$, then $(P,L)$ contains $c_1(m^{-1/2}n^{3/4})^k$ many $k$-cliques in general position. Here, $c_1 = c_1(k,c)$ is some constant. By hypothesis at most $\sqrt{n}$ lines in $L$ are on a common plane, and notice that\begin{equation*}
        m\leq \delta n \iff \sqrt{n} \leq \delta m^{-1}n^{3/2}.
    \end{equation*} So it suffices for us to take $n_0$ large so that $c_1(m^{-1/2}n^{3/4})^k \geq c_1(n^{1/4}/\delta^{1/2})^k \geq 1$.
\end{proof}

\begin{proof}[Proof of Theorem~\ref{main2}]
    We write $m=|P_r(L)|$ and \eqref{guthkatzrich} implies $m \leq O(n^{3/2}/r^2)$. Since $r \geq n^{1/4}\log(n)$, we have $m \leq O(n / \log^2(n))$. If $m < n^{1/2}$, then we have $m < o(n^{3/2}/r^2)$ trivially because $n^{3/2}/r^2 \geq n^{1/2}\log^2(n)$ as a consequence of $r \leq n^{1/2}/\log(n)$. Hence, we have\begin{equation*}
        n^{1/2} \leq m \leq n / \log^2(n).
    \end{equation*}

    Now we fix a real $c > 0$ arbitrarily, and Theorem~\ref{main} produces a constant $\delta = \delta(k,c)$. As long as $n$ is large enough, we can compute\begin{equation*}
        \delta m^{-1} n^{3/2} \geq \delta n^{1/2} \log^2(n) > n^{1/2}.
    \end{equation*} Then we must have $|I(P_r(L),L)| < c m^{1/2}n^{3/4}$, otherwise the property of $\delta$ asserted by Theorem~\ref{main} implies that $(P_r(L),L)$ contains a $k$-clique in general position (provided $n$ is large). Hence, we can solve from the bound $rm \leq |I(P_r(L),L)|$ that $m < c^2 n^{3/2}/r^2$. Our claimed bound on $|P_r(L)|$ follows by taking $c$ to be arbitrarily small.
\end{proof}

\section{Proofs of Theorem~\ref{lowerbound2d} and Theorem~\ref{lowerbound3d}}\label{sec_lower}

For every integer $n\geq 1$,  we can define a dense point-line arrangement $\mathcal{A}^2_n = (P,L)$ in $\mathbb{R}^2$ by \begin{align*}
    P &= \left\{(a,b)|~ a,b\in \mathbb{Z},~ 1\leq a\leq n^{1/3},~ 1\leq b\leq n^{2/3}\right\},\\
    L &= \left\{y=ax+b|~ a,b\in \mathbb{Z},~ 1\leq a\leq n^{1/3},~ 1\leq b\leq n^{2/3}\right\}.
\end{align*} It's well-known that $|I(P,L)| > \Omega(n^{4/3})$. For example, we can check that at least $n/2$ points in $P$ are incident to at least $n^{1/3}/2$ lines in $L$ (see \cite{elekes2001sums} for details).

\begin{proof}[Proof of Theorem~\ref{lowerbound2d}]
    Let $(P', L')$ be randomly sampled from the two-dimensional dense arrangement $\mathcal{A}^2_N$ such that each point and each line is independently chosen with probability $q$. With foresight, we take $N = (2n)^{k/(k-2)}$ and $q = N^{-2/k}$.

    Let $X_1$ be the event that ``at least $qN/4$ points in $P'$ are incident to at least $qN^{1/3}/4$ lines in $L'$''. Inside $\mathcal{A}^2_N$, we can take $N/2$ points, denoted by $Q$, each incident to at least $N^{1/3}/2$ lines. For each $p \in Q \cap P'$, let $L_p$ be $N^{1/3}/2$ lines in $\mathcal{A}^2_N$ incident to $p$.  Since $\mathbb{E}(|Q \cap P'|) = qN/2$ and $\mathbb{E}(|L_p \cap L'|) = qN^{1/3}/2$, we can use Chernoff's inequality (see e.g. Theorem~2.8 in \cite{janson2011random}) to argue that\begin{equation*}
        \Pr\left[|Q \cap P'| < \frac{qN}{4}\right] < \exp\left(\frac{-qN}{24}\right) \quad\text{and}\quad \Pr\left[|L_p \cap L'| < \frac{qN^{1/3}}{4}\right] < \exp\left(\frac{-qN^{1/3}}{24}\right).
    \end{equation*} Hence using the union bound, we have $\Pr[\bar{X_1}] < N \exp({-qN^{1/3}}/{24})$.

    Let $X_2$ be the event that ``$qN/2<|P'|,|L'|<3qN/2$''. Since $\mathbb{E}(|P'|)=\mathbb{E}(|L'|) = qN$, we can check that $\Pr[\bar{X_2}]< 4 \exp(-qN/12)$ using Chernoff's inequality. Let $X_3$ be the event that ``$(P', L')$ contains no $k$-cliques in general position''. Notice that there are at most $N^k$ general position $k$-cliques contained in $\mathcal{A}^2_N$, and each one survives with probability $q^{k+\binom{k}{2}}$ in the random sampling. We can apply Markov's inequality to argue that $\Pr[\bar{X_3}] < N^{-1}$.

    By choosing a large constant in the $O$-notation of the statement we wish to prove, we can assume $k \geq 7$ and hence $\Pr[\bar{X_1}\vee \bar{X_2}\vee \bar{X_3}] \to 0$ as $n\to \infty$. As a consequence of event $X_2$, we have\begin{equation*}
        \min \left\{|P'|,|L'|\right\} > qN/2 = n > \max\left\{|P'|,|L'|\right\}/3.
    \end{equation*} So by a simple averaging argument, we can choose $P\subset P'$ and $L\subset L'$ with $|P|=|L|=n$ such that\begin{equation*}
        |I(P,L)| > |I(P',L')|/9 \geq \frac{q^2N^{\frac{4}{3}}}{144} = \frac{N^{\frac{4}{3}-\frac{4}{k}}}{144} = n^{4/3 - O(1/k)}.
    \end{equation*} Here, the second inequality above is a consequence of event $X_1$. Moreover, $(P, L)$ contains no $k$-cliques in general position as a consequence of event $X_3$.
\end{proof}

For integers $m,n\geq 1$ with $n^{1/2}\leq m\leq n^{3/2}$, we can define a dense point-line arrangement $\mathcal{A}^3_{m,n} = (P,L)$ in $\mathbb{R}^3$ by \begin{align*}
    P &= \left\{ (a,b,c)|~ a,b,c\in \mathbb{Z},~ 1\leq a \leq m^{\frac{1}{2}}/n^{\frac{1}{4}},~ 1\leq b,c\leq m^{\frac{1}{4}}n^{\frac{1}{8}} \right\},\\
    L &= \left\{ \genfrac\{.{0pt}{1}{y=ax + b}{z=cx + d} |~ a,b,c,d\in \mathbb{Z},~ 1\leq a,c \leq n^{\frac{3}{8}}/m^{\frac{1}{4}},~ 1\leq b,d\leq m^{\frac{1}{4}}n^{\frac{1}{8}} \right\}.
\end{align*} We can check that $|I(P,L)| > \Omega(m^{1/2}n^{3/4})$ since at least $m/4$ points in $P$ are incident to at least $\frac{1}{4}n^{3/4}/m^{1/2}$ lines in $L$. Let $\pi$ be an arbitrary plane in $\mathbb{R}^3$. We claim that $\pi$ contains at most $\sqrt{n}$ lines of $L$. Indeed, if $\pi$ is perpendicular to the $x$-axis, it contains no line of $L$. Otherwise, let $\ell\in L$ be contained in $\pi$ chosen with parameters $(a,b,c,d)$. Then there are at most $m^{\frac{1}{4}}n^{\frac{1}{8}}$ choices of $(b,d)$ since the point $(0,b,d)$ must be contained in $\pi$. There are at most $n^{\frac{3}{8}}/m^{\frac{1}{4}}$ choices of $(a,c)$ since $\pi$ is characterized by a linear equation. Hence our claim follows and $\mathcal{A}^3_{m,n}$ is an extremal construction towards \eqref{guthkatz}.

\begin{proof}[Proof of Theorem~\ref{lowerbound3d}]
    Let $(P', L')$ be randomly sampled from the three-dimensional dense arrangement $\mathcal{A}^3_{M, N}$ such that each point and each line is independently chosen with probability $q$. With foresight, we take $M = (2m)^{k/(k-2)}$, $N = (2n)\cdot(2m)^{-2/(k-2)}$, and $q = M^{-2/k}$. By choosing a large constant in the $O$-notation of the statement we wish to prove, we can assume $k \geq 12$ and check that $N^{1/2}\leq M\leq N^{3/2}$ using our hypothesis $n^{1/2}\leq m\leq n$.

    Let $X_1$ be the event that ``at least $qM/8$ points in $P'$ are incident to at least $\frac{q}{4}N^{3/4}/M^{1/2}$ lines in $L'$'', $X_2$ be the event that ``$qM/2<|P'|<3qM/2$ and $qN/2<|L'|<3qN/2$'', and $X_3$ be the event that ``$(P',L')$ contains no $k$-cliques in general position''. By similar arguments as in Theorem~\ref{lowerbound2d}, we can show that $X_1, X_2, X_3$ happens asymptotically almost surely as $m, n \to \infty$. Let $X_4$ be the event that ``any plane contains at most $2q\sqrt{N}$ lines in $L'$''. For any plane $\pi$, if $\pi$ contains more than $q\sqrt{N}$ lines in $L$, then we can apply Chernoff's inequality to argue that\begin{equation*}
        \Pr\left[\text{$\pi$ contains more than $2q\sqrt{N}$ lines in $L'$}\right]<\exp\left(-\frac{2q^2\sqrt{N}}{3}\right).
    \end{equation*} Since there are at most $N^2$ planes containing more than one line in $L$, we conclude that $X_4$ happens asymptotically almost surely.

    Now, by similar arguments as in Theorem~\ref{lowerbound2d} inside a sample where $X_1,X_2,X_3,X_4$ happen simultaneously, we can find $P\subset P'$ and $L\subset L'$ with $|P|=m$ and $|L|=n$ such that the arrangement $(P,L)$ contains no $k$-cliques in general position and $|I(P,L)| > m^{1/2-O(1/k)}n^{3/4}$. Moreover, any plane contains at most $2q\sqrt{N}=o(\sqrt{n})$ lines in $L$ as a consequence of event $X_4$. Hence we conclude the proof.
\end{proof}

\section{Remarks}\label{sec_remark}

\noindent 1. By taking $m=n$ in Theorem~\ref{main}, we can cover the size-symmetric case missed by our Theorem~\ref{main1}, but with a strengthened ``truly three-dimensional'' condition.
\begin{proposition}
     For every integer $k \geq 4$ and real $c > 0$, there exist constants $n_0 = n_0(k,c)$ and $\delta = \delta(k,c)$ such that if an arrangement $(P,L)$ of $n$ points and $n$ lines in $\mathbb{R}^3$ satisfies: $n \geq n_0$; $|I(P,L)| > c n^{5/4}$; and at most $\delta \sqrt{n}$ lines of $L$ lie in a common plane, then $(P,L)$ contains a $k$-clique in general position.
\end{proposition}
\noindent However, this result would only be interesting if there exists an arrangement of $n$ points and $n$ lines with at most $\delta \sqrt{n}$ lines of $L$ in a common plane and $c n^{5/4}$ many incidences. Note that $\delta$ is a very large number depending on $c$. Although the Guth--Katz bound \eqref{guthkatz} doesn't rule out such an arrangement, we believe it is unlikely to exist.

\medskip

\noindent 2. A \textit{$k$-grid} inside an arrangement $(P, L)$ is defined as a pair $(L_1, L_2)$ of disjoint $k$-subsets of $L$ such that $\{\ell_1\cap \ell_2|~ \ell_i\in L_i\}$ consists of $k^2$ distinct points in $P$. Mirzaei and Suk \cite{mirzaei2021grids} proved that any arrangement of $n$ points and $n$ lines in $\mathbb{R}^2$ without $k$-grids has at most $O(n^{4/3 - c_k})$ incidences for some $c_k > 0$. It would be interesting if we could obtain such a polynomial improvement for ``truly three-dimensional'' point-line arrangements in $\mathbb{R}^3$.

We remark that if we forbid $O(1)$ lines of $(P, L)$ in a common plane or a degree $2$ algebraic surface, then we automatically forbid $O(1)$-grids in $(P, L)$. Indeed, suppose $(P, L)$ contains a $k$-grid $(L_1, L_2)$ with $k \geq 3$ and we take $\ell_1, \ell_2, \ell_3 \in L_1$. By the theory of reguli, there is a degree $2$ polynomial $f$ such that $\ell_1,\ell_2,\ell_3 \in Z(f)$ (see e.g. Proposition~8.14 in \cite{guth2016polynomial}). Since a line not in $Z(f)$ has at most $2$ intersections with $Z(f)$, each $\ell \in L_2$ must be in $Z(f)$ as it intersects $\ell_1,\ell_2,\ell_3$ in distinct points.

\bibliographystyle{abbrv}
\bibliography{main}
\end{document}